\documentclass[11pt]{amsart}
\usepackage[utf8]{inputenc}
\usepackage[english]{babel}
\usepackage{graphicx}
\usepackage{amssymb}
\usepackage{mathbbol}
\usepackage{amsmath}
\usepackage[mathscr]{euscript}
\usepackage{accents}
\usepackage{color}

\usepackage{tikz}
\usetikzlibrary{decorations.markings}

\usepackage{microtype} 

\usepackage[
color=orange!80, 
bordercolor=black,
textwidth=3cm,
textsize=small,
colorinlistoftodos]
{todonotes}
\newcommand{\dynote}[1]{\todo[color=red!40,linecolor=red!40!black,size=\tiny]{#1}}

\theoremstyle{plain}	
\newtheorem{theorem}{Theorem}[section]
\newtheorem{proposition}[theorem]{Proposition}
\newtheorem{corollary}[theorem]{Corollary}
\newtheorem{lemma}[theorem]{Lemma}
\theoremstyle{definition} 
\newtheorem{definition}[theorem]{Definition}
\newtheorem{example}[theorem]{Example}
\newtheorem*{pf}{Proof}

\title{Embedded surface invariants via the Broda-Petit construction}
\author{I.J.\ Lee and D.N.\ Yetter }
\date{March 2023}

\begin{document}

\begin{abstract}
We recall Petit's construction of ``dichromatic'' invariants of 4--manifolds computed from Kirby diagrams using a nested pair of ribbon fusion categories ${\mathcal B}\subset{\mathcal C}$ as initial data. Along the way we prove a lemma that places the use of formal linear combinations of simple objects with quantum dimensions as coefficients—as in the constructions of Reshetikhin--Turaev, Broda, and Petit—more firmly in the functorial framework favored by the authors. We then show that Hughes--Kim--Miller's banded--unlink presentations of surfaces embedded in 4--manifolds provide a means whereby a Frobenius algebra in ${\mathcal B}$, together with a suitable module over it lying in ${\mathcal C}$, gives rise to an invariant of a surface--4--manifold pair. We provide a class of examples of suitable initial data and compute examples showing that the invariant is sensitive to both genus and knotting.
\end{abstract}

\maketitle

\section{Introduction}

In the late 1990's B. Broda described a surgical approach to Crane-Yetter theory via Kirby calculus \cite{K}, which he subsequently sought to improve in \cite{Br97} by using only ``bosonic'' labels on the 2-handle attaching curves, while using all the irreducible representations of $U_q(sl_2)$ for $q$ a root of unity on the 1-handle attaching curves.  The {\em a priori} more refined ``dichromatic'' invariant, nonetheless, reduced in the case of simply-connected 4-manifolds to an encoding of signature and Euler characteristic, as did Crane-Yetter theory.

The construction was, however, felt to be of sufficient interest that it was generalized by J. Petit \cite{P2008} and B\"{a}renz and Barrett \cite{BB18} to give what they called ``dichromatic'' invariants of smooth 4-manifolds. The details of the two constructions are slightly different:  Petit uses a nested pair of ribbon fusion categories in imitation of Broda's use of the ``bosonic'' representations as a subcategory of $Rep(U_q(sl_2))$, while B\"{a}renz and Barrett use a functor from a spherical category into a ribbon category, as initial ``coloring'' data.  Recently Costantino, Geer, Ha\"{i}oun and Patreau-Mirand \cite{CGHP23} have given similar a construction of (3+1)-dimensional TQFTs using non-semisimple ribbon categories as initial data, including a surgical construction to which the techniques in the present paper should extend.

The present paper will take Petit's construction as a starting point.  In view of the ``no-go'' theorem of Reutter \cite{Reu20}, which shows that semisimple (3+1)-dimensional TQFTs cannot distinguish non-diffeomorphic smooth structures on homeomorphic 4-manfolds, those who regard the only possibly interesting feature of a (3+1)-dimensional TQFT as sensitivity to smooth structure, should regard this work as a toy model.

We will then use a tangle-theoretic rendering of the work of Hughes, Kim and Miller \cite{HKM19, HKM21}, which gives a Kirby-calculus-type presentation of a 4-manifold with an embedded surface as a framed link diagram with ``dotted'' components as in Kirby \cite{K}, together with a ``banded unknot'', to modify the construction to be sensitive to the presence of a specified embedded surface in the 4-manifold, or more precisely, a specified surface embedded in a 4-dimensional (0,1,2)-handlebody.  We will not consider the immersed case Hughes et al. deal with in \cite{HKM21}, but suggest that the reader will, as we did, find the clearer diagrammatics for the embedded case given in that paper helpful.

Throughout we adopt the convention that whenever a diagram of objects and arrows is used syntactically as a statement, the assertion being made is that the diagram commutes.   Moreover, all 4-manifolds considered are connected.


\section{The Broda-Petit construction}

In this section we review Petit's \cite{P2008} generalization of Broda's \cite{Br97} construction, beginning with the necessary notions from category theory to describe the initial coloring data.


\begin{definition}
A {\em monoidal category} $\mathcal C$ is a category $\mathcal C$ equipped
with a functor $\otimes : {\mathcal C} \times {\mathcal C} \rightarrow {\mathcal C}$
and an object $I$, together with natural isomorphisms 
$\alpha : \otimes (\otimes \times 1_{\mathcal C}) \Rightarrow \otimes (1_{\mathcal C}
\times \otimes)$, $\rho : \otimes I \Rightarrow 1_{\mathcal C}$, and 
$\lambda : I\otimes \Rightarrow 1_{\mathcal C}$, satisfying the
pentagon, triangle, and bigon ($\rho_I = \lambda_I$) coherence conditions
(cf. \cite{CWM}).  A {\em tensor category over $K$}, for $K$ some field $K$, is a monoidal abelian category linear over $K$, with
all $-\otimes X$ and $X\otimes -$ exact. 
\end{definition}

\begin{definition}
An object $X$ in an abelian category over $K$ is {\em simple} if its only subobjects are $0$ and $X$, in which case $hom(X,X)$ is a field extension of $K$.  An abelian category is {\em semi-simple} if every object in it isomorphic to a direct sum of simple objects.  A {\em multifusion category} is a semi-simple tensor category with finitely many isomorphism classes of simple objects.  A {\em fusion category} is a multifusion category in which, moreover the monoidal identity $I$ satisfies ${\rm hom}(I,I) \cong K$ (and is thus {\em a fortiori} simple).
\end{definition}

We will be concerned with monoidal categories with additional structure.

\begin{definition} \label{bmc}
A {\em braided monoidal category} is a monoidal category equipped with a
monoidal natural isomorphism $\sigma: \otimes \Rightarrow \otimes(tw)$,
where $tw:  {\mathcal C} \times {\mathcal C} \rightarrow {\mathcal C} \times {\mathcal C}$
is the ``twist functor'' ($tw(f,g) = (g,f)$).  And satisfying
\vspace*{10pt}

\setlength{\unitlength}{0.011in}%
\begin{center} \begin{picture}(335,161)(50,615)
\thinlines
\put( 90,710){\vector( 4, 3){ 40}}
\put(205,760){\vector( 1, 0){ 75}}
\put(205,640){\vector( 1, 0){ 75}}
\put( 90,685){\vector( 1,-1){ 35}}
\put(345,750){\vector( 1,-1){ 40}}
\put(340,650){\vector( 4, 3){ 40}}
\put( 50,690){\makebox(0,0)[lb]{\raisebox{0pt}[0pt][0pt]
{ $(A \otimes B) \otimes C$}}}
\put(120,755){\makebox(0,0)[lb]{\raisebox{0pt}[0pt][0pt]
{  $A \otimes (B \otimes C)$}}}
\put(295,755){\makebox(0,0)[lb]{\raisebox{0pt}[0pt][0pt]
{  $(B \otimes C) \otimes A$}}}
\put(120,635){\makebox(0,0)[lb]{\raisebox{0pt}[0pt][0pt]
{  $(B \otimes A) \otimes C$}}}
\put(290,635){\makebox(0,0)[lb]{\raisebox{0pt}[0pt][0pt]
{  $B \otimes (A \otimes C)$}}}
\put(365,690){\makebox(0,0)[lb]{\raisebox{0pt}[0pt][0pt]
{  $B \otimes (C \otimes A)$}}}
\put( 85,730){\makebox(0,0)[lb]{\raisebox{0pt}[0pt][0pt]
{  $\alpha$}}}
\put(235,770){\makebox(0,0)[lb]{\raisebox{0pt}[0pt][0pt]
{  $\sigma^{\pm 1}$}}}
\put(380,735){\makebox(0,0)[lb]{\raisebox{0pt}[0pt][0pt]
{  $\alpha$}}}
\put( 50,655){\makebox(0,0)[lb]{\raisebox{0pt}[0pt][0pt]
{  $\sigma^{\pm 1} \otimes C$}}}
\put(230,615){\makebox(0,0)[lb]{\raisebox{0pt}[0pt][0pt]
{  $\alpha$}}}
\put(380,655){\makebox(0,0)[lb]{\raisebox{0pt}[0pt][0pt]
{  $B \otimes \sigma^{\pm 1}$}}}
\end{picture} \end{center}

\noindent Here the same choice of $+1$ or $-1$ exponents must be made in all three instances.

A braided monoidal category is a {\em symmetric monoidal category} if the components of $\sigma$ satisfy $\sigma_{B,A} (\sigma_{A,B}) = 1_{A\otimes B}$ for all objects $A$ and $B$.

\end{definition}

As an aside, many authors write out two hexagons involving $\sigma$.  One is the one given above, the other is parenthesized differently and is equivalent to that above with $\sigma^{-1}$.

Braided monoidal categories are, of course, generalizations of the more classical notion of a symmetric monoidal category in which $\sigma_{B,A} = \Sigma_{A,B}^{-1}$.  It will be important for our purposes to distinguish those object for which this equation holds, even when it does not in general:

\begin{definition}
An object $X$ in a braided monoidal category ${\mathcal C},\otimes, I, \alpha, \rho, \lambda, \sigma$ is {\em transparent} if 

\[ \forall Y\in Ob({\mathcal C}) \;\; \sigma_{X,Y} = \sigma_{Y,X}^{-1}\]

The full subcategory of transparent objects is called the M\"{u}ger center of the category.  In a braided fusion category ${\mathcal C}$, if $S_{\mathcal C}$ is a set of representatives of the isomorphisms of simple objects, the denote the subset of $S_{\mathcal C}$ consisting of transparent objects by $T_{\mathcal C}$
\end{definition}

Because we will be working in a context which there is a distinguished subcategory ${\mathcal B}\subset {\mathcal C}$ of a larger braided category, we will also need

\begin{definition}
Given a subcategory ${\mathcal B}\subset {\mathcal C}$ of a braided category ${\mathcal C}$ an object $X$ of ${\mathcal C}$ if {\em ${\mathcal B}$-transparent} if

\[ \forall Y\in Ob({\mathcal B}) \;\; \sigma_{X,Y} = \sigma_{Y,X}^{-1}\]
\end{definition}

\begin{definition}
A {\em right dual} to an object $X$ in a monoidal category $\mathcal C$ is
an object $X^\ast$ equipped with maps 
$\epsilon: X \otimes X^\ast \rightarrow I$ and
$\eta: I \rightarrow X^\ast \otimes X$ such that the compositions

\[ X \stackrel{\rho^{-1}}{\rightarrow} X \otimes I 
\stackrel{X\otimes \eta}{\rightarrow} 
X \otimes (X^\ast \otimes X) \stackrel{\alpha^{-1}}{\rightarrow} 
(X\otimes X^\ast )\otimes
X \stackrel{\epsilon \otimes X}{\rightarrow} I\otimes X 
\stackrel{\lambda}{\rightarrow} X \]

\noindent and

\[ X^\ast \stackrel{\lambda^{-1}}{\rightarrow} I \otimes X^\ast 
\stackrel{\eta \otimes X^\ast }{\rightarrow} 
(X^\ast \otimes X) \otimes X^\ast \stackrel{\alpha}{\rightarrow} 
X^\ast \otimes 
(X \otimes X^\ast) \stackrel{X^\ast \otimes \epsilon}{\rightarrow} 
X^\ast \otimes I 
\stackrel{\rho}{\rightarrow} X^\ast \]

\noindent are both identity maps.

A {\em left dual} $^\ast X$ is defined similarly, 
but with the objects placed on opposite
sides of the monoidal product.
\end{definition}

This type of duality is an abstraction from the sort of duality which exists
in categories of finite dimensional vector spaces. It is not hard to show
that the canonical isomorphism from the second dual of a vector space
to the space generalizes to give canonical isomorphisms 
$k: ^\ast(X^\ast) \rightarrow X$ and $\kappa: (^\ast X)^\ast \rightarrow X$.
In general, however, there are not necessarily even any maps (or non-zero maps in the linear context) from
$X^{\ast \ast}$ or $^{\ast \ast} X$ to X. (cf. \cite{FY.cohere}).
In cases where every object admits a right (resp. left) dual, it is easy to
show that a choice of right (resp. left) dual for every object extends to
a contravariant functor, whose application to maps will be denoted
$f^\ast$ (resp. $^\ast f$), and that the canonical maps noted above
become natural isomorphisms between the compositions of these functors
and the identity functor.  Likewise, it is easy to show that
$(A\otimes B)^\ast$ is canonically isomorphic to $B^\ast \otimes A^\ast$ 
and similarly for left duals.

\begin{definition}
A {\em pivotal category} is a monoidal category with right duals together with a natural isomorphism $a_X:X\rightarrow X^{\ast\ast}$. 
\end{definition}

It is easy to see that a pivotal structure induces the structure of a left dual on each right dual.

\begin{definition}
Let $X$ be an object in a pivotal category $\mathcal C$. Left and right traces $tr_{L}$, $tr_{R} : \coprod_{X\in Ob{\mathcal C}} \mathcal C ( X , X ) \to \mathcal C ( I , I )$ are be defined as follows:

\[ tr_{R} ( f ) = \epsilon_{X} \,( a_X^{-1}f \otimes 1 ) \,\eta_{X^*} \]
\[ tr_{L} ( f ) = \epsilon_{X^*} \,( 1 \otimes fa_X ) \,\eta_X \]

\end{definition}

\begin{definition}
A pivotal category $\mathcal C$ is {\em spherical} if it satisfies $tr_{L}(f)$ = $tr_{R}(f)$ for all endomorphisms $f$.  In the context of spherical categories, we can thus suppress the distinction between left and right traces, and simply write $tr(f)$.

In any spherical category, the {\em (quantum) dimension} of an object $X$ is the endomorphism of the monoidal identity object given by $dim(X) = tr(1_X)$, the trace of its identity map.  In the case of a spherical fusion category over $K$, it may be regarded as an element of the field $K$.

The {\em dimension} of a spherical fusion category ${\mathcal C}$ is
\[ \Delta_{\mathcal C} := \sum_{X\in S_{\mathcal C}} dim(X)^2 \]

\noindent where $S_{\mathcal C}$ is a set containing one object from each isomorphism class of simple objects in ${\mathcal C}$.
\end{definition}

The following proposition is easy to establish:

\begin{proposition}
    In any fusion category, the family of maps ${\mathsf d}_X:X\rightarrow X$ which multiplies each direct summand of $X$ by its quantum dimension is independent of the choices decompositions into direct summands and forms a natural transformation.
\end{proposition}

The components of ${\mathsf d}$ at the objects $B_{\mathcal C}:=\oplus_{X \in S({\mathcal C})} X$  will be used in our description of Petit's invariants.  It is easy to see that $tr({\mathsf d}_{B_{\mathcal C}} = \Delta_{\mathcal C}$.

In the case of a braided monoidal category every right dual
is also a left dual, in general in a non-canonical way
(cf. \cite{FY.cohere}). In symmetric monoidal
categories, we return to the familiar: right duals are canonically
left duals. For non-symmetric braided monoidal categories, the structure
will be canonical only in the presence of additional structure.  As we will be 
concerned only with the case of braided categories in which all objects admit such two-sided
duals, we make

\begin{definition}
A braided monoidal category $\mathcal C$ is {\em ribbon (or tortile)} 
if all objects
admit right duals, and it is equipped with a natural transformation
$\theta: 1_{\mathcal C} \Rightarrow 1_{\mathcal C}$, which satisfies

\[ \theta_{A\otimes B} = \sigma_{B,A}(\sigma_{A,B}(\theta_A \otimes \theta_B))
\]

\noindent and

\[ \theta_{A^\ast} = \theta_A^\ast \] 

Ribbon fusion categories are also called {\em premodular categories}.
\end{definition}


Any ribbon category is spherical when equipped with left duals whose structure maps are obtained by pre- (resp. post-)composing $epsilon$ (resp. $\eta$) by a component of $\sigma$ (resp. $\sigma^{-1}$).  For details of the canonicity of the correspondence between right
and left duals in the ribbon or tortile case see \cite{Y.FTTD}.
 
The importance of ribbon categories for knot theory and 3- and 4-manifold topology is provided by an extremely beautiful special case of the
 coherence theorem of Shum which characterized ribbon categories freely generated by a small category \cite{Shum}:

\begin{theorem}
The ribbon category freely generated by a single object
$X$ is monoidally equivalent to the category of framed tangles $\mathcal FT$.
\end{theorem}

The relationship between framed links and surgical presentations of 4-manifolds provided by the Kirby calculus \cite{K} then provides the starting point for the construction of 4-manifold invariants given by Broda \cite{Br97} and its generalization by Petit \cite{P2008}.  Shum's full coherence theorem also is the justification for the now commonplace use of knot-theoretic features in string-diagram calculations for maps in ribbon categories, which we will use without further comment.

Before recalling the construction of Petit \cite{P2008} in sufficient detail for our purposes, we prove a purely knot-theoretic lemma, which we feel provides a  wholely satisfactory explanation of why functorial knot invariants can be modified by introducing numerical coefficients (of quantum dimensions) on direct summands of the object to which a strand in the knot diagram is mapped has been done in \cite{RT, Br97, P2008}:

\begin{theorem} \label{jumprope} Suppose $L_0$ and $L_1$ are two equivalent links in ${\mathbb R}^3$,
and suppose, moreover, that each component of $L_0$ and $L_1$ has a 
unique global 
minimum.  Then there
is an ambient isotopy $H$ which preserves the component-wise global minima 
throughout the isotopy in the sense that if $S_t$ denotes the set of
component-wise
global minima on $L_t = H(L_0,t)$, then $S_t = H(S_0,t)$ for all 
$t \in [0,1]$. 
\end{theorem}

\begin{proof}  
Begin by modifying each of $L_0$ and $L_1$ using 
an isotopy by $\Delta$-moves, in the
sense of Reidemeister \cite{Rei32}, to move the component-wise global minima 
below the plane $z = -1$ and leave the rest of each link
  lying between $z=0$ and $z=1$, except for edges incident 
with the global minima.
Plainly, this can be done without moving any points below the minima of their 
respective components.

It thus suffices to see that there is an isotopy between the modified links
$L_0^\prime$ and $L_1^\prime$ with the desired properties. Mark each of the 
global minima with a $\wedge$ on the curve.  Now view the links
in projection onto the $xy$-plane along a vector arbitrarily close to vertical 
for which the projection is a knot diagram with only transverse double points.

Now, a sequence of Reidemeister moves and isotopies of the projection (all
corresponding to $\Delta$-moves in ${\mathbb R}^3$) will turn the one diagram into
the other, including sending the marked minima to each other.

The problem is that some of these moves may correspond to moves in ${\mathbb R}^3$
which lower points below the global minimum of the component on which they lie.

Specifically, instances of the Reidemeister II (resp., III) move in which an
arc (resp., one or two arcs incident with a crossing) passes under another 
arc containing a minimum, and isotopies of the projection in which an
undercrossing arc moves from one side of a minimum to the other correspond
to portions of an isotopy which may violate the preservation of global minima.
When the arc passed under belongs to a different component from the minimum,
this may not violate the condition. If it does, it is easily remedied:
lower the minimum of the underpassing component well below the minimum
being passed under before doing the move.

In the case where the undercrossing arc belongs to the same component, we need
to do more work:
with the minimum set well below undercrossing arc, the move on 
the diagram corresponds to a
sequence of $\Delta$-moves, most of which simply lower the offending arc 
without 
changing the diagram,
one of which passes it below what should have been the global minimum, and
the rest of which bring the moved arc back up into the region where
the rest of the link is to lie at the end of the isotopy.

The ``before'' and ``after'' positions of the arc changed by the move form a loop
which we may assume lies between $z=0$ and $z=1$.  Lower the global min
of the component involved below all of the other global mins by 
an isotopy fixing the link above $z=0$.  (Doing this by $\Delta$-moves
without violating the global minimum preservation may require 
briefly introducing
a vertical segment projecting to the same point, but this is resolved 
immediately
by another $\Delta$-move.)

Now, complete the disk across which the errant move was made to a polyhedral
sphere, intersecting the initial and final links only in the segments changed
by the move, and such that the complementary disk forming the sphere 
lies entirely above the just-lowered
minimum.  (The new faces will slope down and slightly away from the boundary 
of the old disk to a level between the other minima and the lowered one
on the component being modified, then level off to form a 
``floor'' extending beyond the link in all horizontal directions, ``walls'' 
extending
above the link, and a ``ceiling'' completing the polyhedral sphere above
the link.)  The move can be replaced with a sequence of $\Delta$-moves
across the complementary disk of the sphere.

Thus, no moves which fail to preserve the global minima are necessary and the 
theorem holds. 
\end{proof}

A convenient way to visualize the last step in the proof is to imagine the component on which the offending move is to be made as a person holding a jump-rope (resp. a pair of crossed jump-ropes for the Reidemeister III case), and wanting to get it (resp. them) from in front of their feet to behind.  The offending Reidemeister move involves jumping, to bring the rope(s) under the feet, but the same position is obtained by swinging the ropes over one's head and bringing them down behind the feet.

As an aside, the theorem very much depends on the availability of isotopies which raise and lower the component-wise global minima.  If the minima were fixed, one could assume they are fixed on a ``sphere at infinity'' and one would be effectively dealing with isotopy classes of string links, which it is well-known are are not equivalent to isotopy classes of links.

\begin{corollary}
Any ambient isotopy invariant of (framed) links equipped with a distinguished base-point
on each component gives rise to an ambient isotopy invariant of links by choosing the
global minimum with respect to some coordinate direction to be the base-point.  Moreover, the same is true for isotopy invariants of colored (framed) links (that is links in which the components are divided into disjoint subsets designated by different colors).
\end{corollary}

Note that the choice of base-points or coordinates is really irrelevant: given
base-points, one could first drag them down in the coordinate direction 
to make them
the global minima, or, given minima with respect to a different coordinate
direction, one could first rotate a the link so that the ``downward'' 
directions align;
then use the procedure of the proof to isotope one to the other preserving
the minima.

This lemma then gives a satisfactory explanation of how to introduce in a functorial way the coefficients on which the constructions of 3-manifold invariants of Reshetikhin and Turaev \cite{RT} and 4-manifold invariants of Broda \cite{Br97} and Petit \cite{P2008} depend:  the base point, regarded as an endomorphism of the object whose identity arrow is denoted as a descending arc, is mapped to the map we have denoted ${\mathsf d}$ which multiplies the direct summands of the coloring object by their quantum dimensions.





We now redescribe the construction of Petit's dichromatic invariants of 4-manifolds, making use of the foregoing, beginning by recalling the usual notions from the Kirby calculus \cite{K}.

\begin{definition}
A {\em Kirby diagram} is a blackboard framed diagram of an oriented framed link $L$, with a distinguished subset of its components forming a 0-framed unlink, $L_0$.
\end{definition}

Traditionally the components of $L_0$ are indicated by placing a large dot somewhere on the component.

From Kirby \cite{K} recall that such a diagram describes a 4-dimensional 0-1-2-handlebody, by regarding it as lying in the bounding $S^3$ of a 4-ball $B^4$ (the 0-handle), regarding the components of $L_0$ as locations from which 2-handles are to be removed from the interior of $B^4$, thereby effectively attaching a 1-handle which would be cancelled by the replacement of the removed 2-handle, and the components of $L\setminus L_0$ as giving instructions for the attaching of 2-handles.  Moreover, if the boundary of the resulting handlebody is diffeomorphic to the boundary of a 0-1-handlebody, there will be a unique way to complete the handlebody with 3-handles and a 4-handle to give a closed 4-manifold, and all closed connected smooth 4-manifolds arise in this way.

Also, as shown in \cite{K}, two such diagrams describe diffeomorphic handlebodies (closed 4-manifolds, when the boundary is appropriate) if and only if they are related by a sequence of (1) isotopies (which, using the blackboard framing convention are encoded by the framed Reidemeister moves, the usual first Reidemeister move being replaced with a move which does not remove the curl, but moves it to the other side of the strand, keeping the same crossing), (2) handle slide moves of one component over another, in which the component ``slid over'' the other is replaced with its band sum with a parallel copy of the the component over which it was slid, pushed off in the direction of the framing, with slides of components in $L_0$ over components not in $L_0$ being forbidden, as they do not correspond to a geometric handle-slide, and (3) stablilization moves:  addition or removal of a 0-framed component linked with a dotted component, or (for closed 4-manifolds) addition or removal of a 0-framed component, in either case, lying in a disk in the plane of the diagram which is disjoint from the rest of the diagram.

\begin{definition}
A {\em based Kirby diagram} is a Kirby diagram with a base point chosen on each component, lying on a strand oriented downward.
\end{definition}

In a category of tangles, a marked point on a strand may be taken to be an endomorphism of the object whose identity map an unmarked strand represents.  Doing this allows us to make the usual surgical constructions of 3- and 4-manifold invariants via the Kirby calculus functorial.

Consider the free ribbon category ${\mathcal K}$ generated by a category with two objects, ${\bf 1}$ and ${\bf 2}$, each of which have a copy of $({\mathbb N},+,0)$ as their set of endomorphisms, these being the only arrows in the generating category.  An easy application of Shum's \cite{Shum} full coherence theorem shows that this is equivalent to the category of framed tangles with marked points (with no limit on how many marked points occur or where they occur).

Based Kirby diagrams can be regarded as forming a particular sub-monoid (under separated union) of the commutative monoid of endomorphisms of the monoidal identity ${\mathcal K}$ -- namely that sub-monoid whose elements, when written as string diagrams, have exactly one marked point on each connected component, and with connected components colored ${\bf 1}$ forming a 0-framed unlink.

The effects of Petit's ``Kirby colors'' are then obtained by applying the functor induced by freeness, mapping ${\bf 1}$ (resp. ${\bf 2}$) to the object $B_{\mathcal C}:=\oplus_{X \in S({\mathcal C})} X$, (resp. $B_{\mathcal B}:=\oplus_{X \in S({\mathcal B})} X$), and the marked point ($1 \in {\mathbb N}$ as the generating endomorphism of the object), to ${\mathsf d}_{B_{\mathcal C}}$ (resp. ${\mathsf d}_{B_{\mathcal B}}$), where $S({\mathcal C})$ (resp. $S({\mathcal B})$) is a set of representative objects for the isomorphism classes of simple objects in ${\mathcal C}$ (resp. ${\mathcal B}$).

We call this functor ``the Kirby functor'', denote it by $\big< - \big>$, observe that as $\mathcal C$ is assumed to be a fusion category, its values on based Kirby diagrams are scalar multiples of the identity arrow on the monoidal identity, with which scalars we identify the images of based Kirby diagrams under the functor. 

By Shum's coherence theorem, the values of the Kirby functor are isotopy invariant.  For invariance under handle-slides Petit \cite{P2008} relies on results of Brugi\`{e}res \cite{Br00}, but a clearer proof of the necessary invariance properties can be found by applying Lemma 3.3 of B\"{a}renz and Barrett \cite{BB18} to the identity functor on ${\mathcal C}$ and the inclusion functor of ${\mathcal B}$ into ${\mathcal C}$, which makes obvious why the sliding property holds, even if the link-component slid over is knotted, non-trivally framed or linked with other components. 

We will also have call to use the generalization of the Lickorish encirclement property given by Brugi\`{e}res \cite{Br00}:

\begin{lemma} \label{encircle}
Let $\mathcal{C}$ be a ribbon fusion category with $\Delta_{\mathcal{C}} \in \mathbb{k}^{*}$. For every object $X$ of $\mathcal{C}$, we have:
\begin{eqnarray*}
    \lefteqn
(\;\epsilon_{B_{\mathcal C}^*} \otimes 1_X )( 1 \otimes {\mathsf d}_{B_{\mathcal C}} \otimes 1_X) (1_{B_{\mathcal C}^*} \otimes \sigma_{B_{\mathcal C},X}\sigma_{X,B_{\mathcal C}})(\eta_{B_{\mathcal C}} \otimes 1_X) \\
 =  \left\{ \begin{array}{cl} \Delta_{\mathcal{C}}\,\, 1_{X} & \textrm{if $X$ is a transparent object}, \\ 0 & \textrm{otherwise}. \end{array} \,\right.
\end{eqnarray*}

for every simple object $X$.\\
\end{lemma}

The map in Lemma \ref{encircle} is given in string diagram notation by

\[\includegraphics[width=1in]{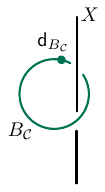}\]

Restated in our favored terminology, the main result of Petit \cite{P2008} regarding 4-manifolds is then:

\begin{theorem}
Let $\mathcal{B} \subset \mathcal{C}$ be ribbon fusion categories such that their dimensions  and the quantity
\[\Delta_{\mathcal{C}, \mathcal{B}}^{\prime \prime} = \sum_{X \in S_{\mathcal{B}} \cap T_{\mathcal{C}}} \textrm{dim} (X)^2 \]
are invertible, $W$ a closed, compact, connected, and oriented 4-manifold, and $L$ a based Kirby diagram which represents $W$, then the following scalar is independent of the choice of the based Kirby diagram, and thus an invariant of the 4-manifold $W$,
    \[ I_+ (W) = \frac{\big<L\big>}{(\Delta_{\mathcal{B}})^{b_0(L)} \, (\Delta_{\mathcal{C}} \, \Delta_{\mathcal{C}, \mathcal{B}}'')^{b_+(L)}}, \]
    
\noindent where $\displaystyle \Delta_{\mathcal{C}, \mathcal{B}}^ {\prime \prime} = \sum_{X \in S_{\mathcal{B}} \cap T_{\mathcal{C}}} \textrm{dim} (X)^2$, and
$b_\epsilon (L)$ be the number eigenvalues of the linking matrix of $L$ with the sign $\epsilon \in \{0, -, +\}.$
\end{theorem}

To prove that $I_{+}$ is an invariant of closed and oriented 4-manifolds, we need to prove that $I_{+}$ is invariant under the Kirby moves. By construction of Kirby functor, we know that $\big<L\big>$ is invariant under isotopies of framed link and handle slides. In addition, the linking matrix of $L_0$ does not change under isotopies. Therefore, $I_+$ is invariant under isotopies.

Separated union with a 0-framed component plainly increases the number multiplicity of zero as an eigenvalue of the linking matrix by 1 and multiplies 
$\big<L\big>$ by $\Delta_{\mathcal B}$, while separated union with a Hopf link with one component 0-framed and one dotted component increases the number of positive eigenvalues by one and, applying Lemma \ref{encircle} to evaluate the functor on the Hopf link,  multiplies $\big<L\big>$ by $\Delta_{\mathcal{C}} \, \Delta_{\mathcal{C}, \mathcal{B}}''$, so plainly $I_+(W)$ is invariant under the stabilization moves.

As an aside we note that the powers on the normalization factors in the denominator could have been chosen differently.  In particular, instead of $b_+(L)$, the rank of the linking matrix or the number of dotted components could have been used, while $b_0(L)$ could be replace with the number of 2-handle attaching curves minus the number of dotted components.  In fact, in the present work in which after introducing surfaces, we obtain only invariants of surfaces in (0,1,2)-handlebodies, the first factor in the denominator could be omitted entirely. Different normalizations may be preferable in various contexts, but we will not examine the question here and simply use Petit's choice.

\section{Generalized Broda-Petit invariants for 4D (0,1,2)-handlebody/surface pairs}

In \cite{HKM19} Hughes at al. described an extension of the Kirby calculus which provided an encoding of smooth 4-manifolds with an embedded surface.  This encoding is provided by a more complicated diagrammatics similar to that of Kirby \cite{K}:  besides the framed link encoding the 2-handle attachments and the dotted 0-framed unlink encoding the 1-handles as locations to bore out a 2-handle, their diagrams have another (0-framed, since we are using blackboard framings) unknot, between whose components there are bands.  

This ``banded unknot'' encodes a surface by regarding the unlink as lying in a 3-sphere just inside the bounding 3-sphere of the 0-handle of the 4-manifold, thought of as the standard 4-ball, and having disks hanging inward toward the origin, as the 0-handles of the surface, the bands should be thought of as rising slightly outward, and are the 1-handles of the surface.  The 2-handles are then disks attached to the circles formed by the closure of the symmetric difference between union of the band boundaries and the unlink.  The requirements that these disks exist is a condition on the banded unlink, which in particular implies that its components have linking number zero with those of the dotted unlink encoding the 1-handles.

Moreover, Hughes et al. show that two presentations in this form of 4-manifolds with embedded surfaces give diffeomorphic pairs if and only if they are related by a sequence of moves of the following types:

\begin{description}
\item[$\bullet$ isotopy] sequences of framed Reidemeister moves;
\item[$\bullet$ handle slides] encoded as replacing the sliding component with its band-sum with the component slide over, with the same restrictions as in the Kirby calculus \cite{K};
\item[$\bullet$ Kirby stablization]  separated union with a zero-framed 2-handle attaching curve, or with a Hopf link formed of a 0-framed 2-handle attaching curve and a dotted unknot, or undoing such a move;
\item[$\bullet$ cap]
\[\includegraphics[width=1.2in]{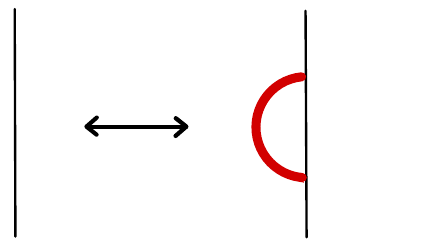}\]
\item[$\bullet$ cup]
  \[\includegraphics[width=1.2in]{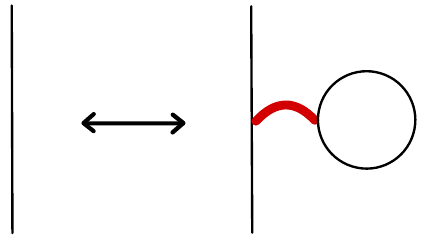}\]
\item[$\bullet$ slides of surface] encoded as band sum moves of a band over any component of the Kirby diagram, or of an unlink component over a dotted component of the Kirby diagram

\item[$\bullet$ band slide]
\[\includegraphics[width=1.2in]{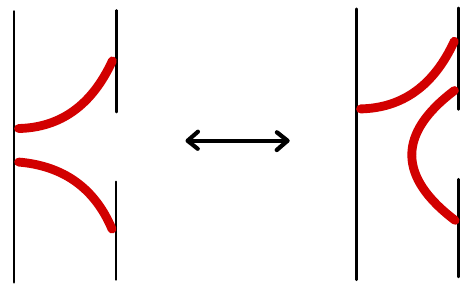}\]
\item[$\bullet$ band swim]
\[\includegraphics[width=1.5in]{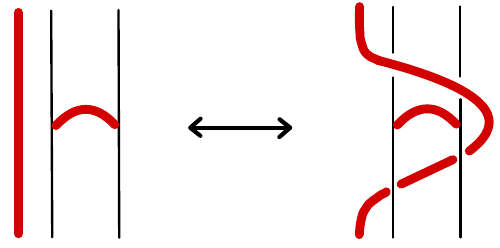}\]
\item[$\bullet$ band/2-handle swim]
\[\includegraphics[width=1.5in]{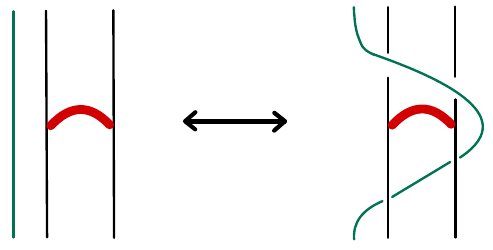}.\]

\end{description}

In the moves shown, rather than described, above, the black arcs are part of the unknot encoding the surface, the thick red arcs are bands in the surface encoding, and the green arc is part of a 2-handle attaching curve.

The reader may note that we have drawn the last three moves in a slightly different position than Hughes et al. do in \cite{HKM21}.  Doing this makes their translation into an algebraic structure in a ribbon category more obvious.

The key to translating the moves into algebraic form is the band slide move:  the geometric justification for it includes an intermediate step in which the upper foot of the lower band in our diagram above in attached to the middle of the upper band.  Readers familiar with the development of algebraic approaches to topological field theory will recognize that if the band (when running vertically) is thought of as the identity arrow on an object in a ribbon category, that object must have the structure of a symmetric Frobenius algebra with the multiplication (or comultiplication, depending on the exact geometry) given by the intermediate step in the band slide move (cf. \cite{L06}):

\begin{definition}
  In any monoidal category, {\em a Frobenius algebra (object)} is an object $F$, equipped with an associative, unital multiplication $\mu:F\otimes F\rightarrow F$ (which we will write as the null infix when applied to elements of underlying sets or vector spaces, when the fusion category admits a faithful underlying functor), with unit $e:I\rightarrow F$ and a coassociative, counital comultiplication $\Delta: F \rightarrow F\otimes F $, with counit map $\epsilon:F \rightarrow I$, and moreover satisfying 
   \[ (1\otimes \mu)(\Delta \otimes 1) = \Delta \mu = (\mu \otimes 1)(1 \otimes \Delta) \]
   \noindent where we have suppressed mention of the associator by MacLane's coherence theorem.
\end{definition}

And, more than that, the state of the band slide move shown on the right, together with this intermediate state requires that the object whose identity arrow is a vertical part of the unlink which is part of the surface encoding must be a module over this Frobenius algebra.  Here a little care is needed:  should it be a left module, a right module or a bimodule?  To have a satisfactory answer, we will require that our embedded surfaces be oriented.  This is accomplished by orienting the unlink encoding the surface, and requiring that this orientation extend compatibly to the curve bounding the upper disks in the description given above.  The surface can then be given the orientation such that the orientation on the 0-handles has the given orientation on the unlink as boundary orientation.

Thus, if we perform isotopies so that bands have no half-twists, a band with feet on vertical sections of the unknot must run from an upward oriented section to a downward oriented section.   Moreover, since the unknot has no thickness, by performing an isotopy that lifts part of the band over (or pushes part of the band under) an unknot component, we may assume that the band always approaches downward (resp. upward) oriented parts of the unknot from the left (resp. right).

From this and the band slide move with its intermediate step included, we see that a band foot, with the band approaching the unknot from above (resp. below), is a string diagram representation of an action (resp. coaction) of a Frobenius algebra on a module, this being a left module if the bit of unknot is oriented downward, and a right module if it is oriented upward:

\begin{definition}
Let $M$ be a left module over $F$ with a left action,
\[ \mu: F \otimes M \rightarrow M \,\,\, \textrm{by} \,\,\, a \otimes m \mapsto a m \]
satisfying

\begin{itemize}
        \item $ a ( b m ) = ( a b ) m $
        \item $(1_F) m = m $
\end{itemize}

Right modules are defined similarly.

We can depict it using red lines as follows for $F$, and black for $M$; for $ a , b \in F , m \in M ,$
\\

\hspace{3mm} \begin{tikzpicture}[thick,scale=0.5]
	\node[red, circle, fill, inner sep=0pt, outer sep=0pt] (1) at (0,0){};
	\node[circle, fill, inner sep=0pt, outer sep=0pt] (2) at (2,0){};
	\node[red, circle, fill, inner sep=0pt, outer sep=0pt] (3) at (0,-1){};
	\node[circle, fill, inner sep=0pt, outer sep=0pt] (4) at (2,-1){};
	\node[circle, fill, inner sep=0pt, outer sep=0pt] (5) at (1,-2){};
	\node[circle, fill, inner sep=0pt, outer sep=0pt] (6) at (1,-3){};
	\node[inner sep=0pt, outer sep=0pt] (7) at (0,0.4){$a$};
	\node[inner sep=0pt, outer sep=0pt] (8) at (2,0.4){$m$};
	\node[inner sep=0pt, outer sep=0pt] (8) at (1,-3.4){$a m$};
	\node[inner sep=0pt, outer sep=0pt] (9) at (9,-1.3){$\mu: F \otimes M \rightarrow M \,\,\,\,\, \textrm{by} \,\,\,\,\, a \otimes m \mapsto a m$};
	
	\draw [ultra thick, red, rounded corners] (1) -- (3);
	\draw [ultra thin, rounded corners] (2) -- (4);
	\draw [ultra thick, red, rounded corners] (3) -- (5);
	\draw [ultra thin, rounded corners] (4) -- (5);
	\draw [ultra thin, rounded corners] (5) -- (6);
\end{tikzpicture}
\\
\\

\hspace{3mm} \begin{tabular}{ll}
\begin{tikzpicture}[thick,scale=0.5]
	\node[red, circle, fill, inner sep=0pt, outer sep=0pt] (1) at (0,0){};
	\node[red, circle, fill, inner sep=0pt, outer sep=0pt] (2) at (1,0){};
	\node[circle, fill, inner sep=0pt, outer sep=0pt] (3) at (3,0){};
	\node[red, circle, fill, inner sep=0pt, outer sep=0pt] (4) at (1,-1){};
	\node[circle, fill, inner sep=0pt, outer sep=0pt] (5) at (3,-1){};
	\node[circle, fill, inner sep=0pt, outer sep=0pt] (6) at (2,-2){};
	\node[red, circle, fill, inner sep=0pt, outer sep=0pt] (7) at (0,-3){};
	\node[circle, fill, inner sep=0pt, outer sep=0pt] (8) at (2,-3){};
	\node[circle, fill, inner sep=0pt, outer sep=0pt] (9) at (1,-4){};
	\node[circle, fill, inner sep=0pt, outer sep=0pt] (10) at (1,-5){};
	\node[inner sep=0pt, outer sep=0pt] (11) at (0,0.4){$a$};
	\node[inner sep=0pt, outer sep=0pt] (12) at (1,0.4){$b$};
	\node[inner sep=0pt, outer sep=0pt] (13) at (3,0.4){$m$};
	\node[inner sep=0pt, outer sep=0pt] (14) at (1,-5.4){$a b m$};
	\node[inner sep=0pt, outer sep=0pt] (15) at (5,-2.5){$=$};
	
	\draw [ultra thick, red, rounded corners] (1) -- (7);
	\draw [ultra thick, red, rounded corners] (2) -- (4);
	\draw [ultra thin, rounded corners] (3) -- (5);
	\draw [ultra thick, red, rounded corners] (4) -- (6);
	\draw [ultra thin, rounded corners] (5) -- (6);
	\draw [ultra thin, rounded corners] (6) -- (8);
	\draw [ultra thick, red, rounded corners] (7) -- (9);
	\draw [ultra thin, rounded corners] (8) -- (9);
	\draw [ultra thin, rounded corners] (9) -- (10);
\end{tikzpicture}
 & \hspace{4mm} 
\begin{tikzpicture}[thick,scale=0.5]
	\node[red, circle, fill, inner sep=0pt, outer sep=0pt] (1) at (0,0){};
	\node[red, circle, fill, inner sep=0pt, outer sep=0pt] (2) at (2,0){};
	\node[circle, fill, inner sep=0pt, outer sep=0pt] (3) at (3,0){};
	\node[red, circle, fill, inner sep=0pt, outer sep=0pt] (4) at (0,-1){};
	\node[red, circle, fill, inner sep=0pt, outer sep=0pt] (5) at (2,-1){};
	\node[red, circle, fill, inner sep=0pt, outer sep=0pt] (6) at (1,-2){};
	\node[red, circle, fill, inner sep=0pt, outer sep=0pt] (7) at (1,-3){};
	\node[circle, fill, inner sep=0pt, outer sep=0pt] (8) at (3,-3){};
	\node[circle, fill, inner sep=0pt, outer sep=0pt] (9) at (2,-4){};
	\node[circle, fill, inner sep=0pt, outer sep=0pt] (10) at (2,-5){};
	\node[inner sep=0pt, outer sep=0pt] (11) at (0,0.4){$a$};
	\node[inner sep=0pt, outer sep=0pt] (12) at (2,0.4){$b$};
	\node[inner sep=0pt, outer sep=0pt] (13) at (3,0.4){$m$};
	\node[inner sep=0pt, outer sep=0pt] (14) at (2,-5.4){$a b m$};
	
	\draw [ultra thick, red, rounded corners] (1) -- (4);
	\draw [ultra thick, red, rounded corners] (2) -- (5);
	\draw [ultra thin, rounded corners] (3) -- (8);
	\draw [ultra thick, red, rounded corners] (4) -- (6);
	\draw [ultra thick, red, rounded corners] (5) -- (6);
	\draw [ultra thick, red, rounded corners] (6) -- (7);
	\draw [ultra thick, red, rounded corners] (7) -- (9);
	\draw [ultra thin, rounded corners] (8) -- (9);
	\draw [ultra thin, rounded corners] (9) -- (10);
\end{tikzpicture}
\\
\\
 \end{tabular}

\hspace{3mm} \begin{tabular}{ll}
\begin{tikzpicture}[thick,scale=0.5]
	\node[red, circle, fill, inner sep=0pt, outer sep=0pt] (1) at (0,0){};
	\node[circle, fill, inner sep=0pt, outer sep=0pt] (2) at (2,1){};
	\node[red, circle, fill, inner sep=0pt, outer sep=0pt] (3) at (0,-1){};
	\node[circle, fill, inner sep=0pt, outer sep=0pt] (4) at (2,-1){};
	\node[circle, fill, inner sep=0pt, outer sep=0pt] (5) at (1,-2){};
	\node[circle, fill, inner sep=0pt, outer sep=0pt] (6) at (1,-3){};
	\node[inner sep=0pt, outer sep=0pt] (7) at (1,-3.7){$(1_F) m$};
	\node[inner sep=0pt, outer sep=0pt] (8) at (4,-1.5){$=$};
	
	\draw [ultra thick, red, rounded corners] (1) -- (3);
	\draw [ultra thin, rounded corners] (2) -- (4);
	\draw [ultra thick, red, rounded corners] (3) -- (5);
	\draw [ultra thin, rounded corners] (4) -- (5);
	\draw [ultra thin, rounded corners] (5) -- (6);
\end{tikzpicture}

 & \hspace{4mm} 
\begin{tikzpicture}[thick,scale=0.5]
	\node[circle, fill, inner sep=0pt, outer sep=0pt] (1) at (0,1){};
	\node[circle, fill, inner sep=0pt, outer sep=0pt] (2) at (0,-3){};
	\node[inner sep=0pt, outer sep=0pt] (7) at (0,-3.4){$m$};
	
	\draw [ultra thin, rounded corners] (1) -- (2);
\end{tikzpicture}
\\
\\
 \end{tabular}

 The corresponding equations for right modules have string diagram representations which are mirror images of these in a vertical line of reflection.

A module $M$ over $F$ always admits a comodule structure on the same side, the coaction being give by precomposing the action by the monoidal product of the copairing with the identity map on the module.  With this coaction the equations in string diagrams below always hold:

\begin{center}
\begin{tabular}{lll}
\begin{tikzpicture}[thick,scale=0.5]
	\node[red, circle, fill, inner sep=0pt, outer sep=0pt] (1) at (0,0){};
	\node[circle, fill, inner sep=0pt, outer sep=0pt] (2) at (3,0){};
	\node[red, circle, fill, inner sep=0pt, outer sep=0pt] (3) at (0,-1){};
	\node[red, circle, fill, inner sep=0pt, outer sep=0pt] (4) at (-1,-2){};
	\node[red, circle, fill, inner sep=0pt, outer sep=0pt] (5) at (1,-2){};
	\node[red, circle, fill, inner sep=0pt, outer sep=0pt] (6) at (1,-3){};
	\node[circle, fill, inner sep=0pt, outer sep=0pt] (7) at (3,-3){};
	\node[circle, fill, inner sep=0pt, outer sep=0pt] (8) at (2,-4){};
	\node[red, circle, fill, inner sep=0pt, outer sep=0pt] (9) at (-1,-5){};
	\node[circle, fill, inner sep=0pt, outer sep=0pt] (10) at (2,-5){};
	\node[inner sep=0pt, outer sep=0pt] (11) at (5,-2.5){$=$};
	
	\draw [ultra thick, red, rounded corners] (1) -- (3);
	\draw [ultra thick, red, rounded corners] (3) -- (4);
	\draw [ultra thick, red, rounded corners] (3) -- (5);
	\draw [ultra thick, red, rounded corners] (5) -- (6);
	\draw [ultra thick, red, rounded corners] (6) -- (8);
	\draw [ultra thick, red, rounded corners] (4) -- (9);
	\draw [ultra thin, rounded corners] (2) -- (7);
	\draw [ultra thin, rounded corners] (7) -- (8);
	\draw [ultra thin, rounded corners] (8) -- (10);
\end{tikzpicture}
&
\begin{tikzpicture}[thick,scale=0.5]
	\node[red, circle, fill, inner sep=0pt, outer sep=0pt] (1) at (0,0){};
	\node[circle, fill, inner sep=0pt, outer sep=0pt] (2) at (2,0){};
	\node[circle, fill, inner sep=0pt, outer sep=0pt] (3) at (1,-1){};
	\node[circle, fill, inner sep=0pt, outer sep=0pt] (4) at (1,-4){};
	\node[red, circle, fill, inner sep=0pt, outer sep=0pt] (5) at (0,-5){};
	\node[circle, fill, inner sep=0pt, outer sep=0pt] (6) at (2,-5){};
	\node[inner sep=0pt, outer sep=0pt] (7) at (3,-2.5){$=$};
	
	\draw [ultra thick, red, rounded corners] (1) -- (3);
	\draw [ultra thin, rounded corners] (2) -- (3);
	\draw [ultra thin, rounded corners] (3) -- (4);
	\draw [ultra thick, red, rounded corners] (4) -- (5);
	\draw [ultra thin, rounded corners] (4) -- (6);
\end{tikzpicture}
&
\hspace{6mm}\begin{tikzpicture}[thick,scale=0.5]
	\node[red, circle, fill, inner sep=0pt, outer sep=0pt] (1) at (-1,0){};
	\node[circle, fill, inner sep=0pt, outer sep=0pt] (2) at (2,0){};
	\node[red, circle, fill, inner sep=0pt, outer sep=0pt] (3) at (2,-1){};
	\node[red, circle, fill, inner sep=0pt, outer sep=0pt] (4) at (1,-2){};
	\node[red, circle, fill, inner sep=0pt, outer sep=0pt] (5) at (3,-2){};
	\node[red, circle, fill, inner sep=0pt, outer sep=0pt] (6) at (-1,-3){};
	\node[circle, fill, inner sep=0pt, outer sep=0pt] (7) at (1,-3){};
	\node[circle, fill, inner sep=0pt, outer sep=0pt] (8) at (0,-4){};
	\node[red, circle, fill, inner sep=0pt, outer sep=0pt] (9) at (0,-5){};
	\node[circle, fill, inner sep=0pt, outer sep=0pt] (10) at (3,-5){};
	
	\draw [ultra thick, red, rounded corners] (1) -- (6);
	\draw [ultra thin, rounded corners] (2) -- (3);
	\draw [ultra thick, red, rounded corners] (3) -- (4);
	\draw [ultra thin, rounded corners] (3) -- (5);
	\draw [ultra thick, red, rounded corners] (4) -- (7);
	\draw [ultra thick, red, rounded corners] (6) -- (8);
	\draw [ultra thick, red, rounded corners] (7) -- (8);
	\draw [ultra thick, red, rounded corners] (8) -- (9);
	\draw [ultra thin, rounded corners] (5) -- (10);
\end{tikzpicture}
\\
\\
 \end{tabular}
\end{center}

\end{definition}

When it is recalled that duality extends to a contravariant functor and that Frobenius algebra objects are canonically self-dual, for $F$ a Frobenius algebra object and $M$ any left module, we always have the following maps:
\\

\begin{center}
\begin{tabular}{ll}

\begin{tikzpicture}[thick,scale=0.5]
	\node[red, circle, fill, inner sep=0pt, outer sep=0pt] (1) at (0,0){};
	\node[red, circle, fill, inner sep=0pt, outer sep=0pt] (2) at (2,0){};
	\node[red, circle, fill, inner sep=0pt, outer sep=0pt] (3) at (0,-1){};
	\node[red, circle, fill, inner sep=0pt, outer sep=0pt] (4) at (2,-1){};
	\node[red, circle, fill, inner sep=0pt, outer sep=0pt] (5) at (1,-2){};
	\node[red, circle, fill, inner sep=0pt, outer sep=0pt] (6) at (1,-3){};
	\node[inner sep=0pt, outer sep=0pt] (7) at (6,-1.3){$\mu: F \otimes F \rightarrow F$};

	\draw [ultra thick, red, rounded corners] (1) -- (3);
	\draw [ultra thick, red, rounded corners] (2) -- (4);
	\draw [ultra thick, red, rounded corners] (3) -- (5);
	\draw [ultra thick, red, rounded corners] (4) -- (5);
	\draw [ultra thick, red, rounded corners] (5) -- (6);
\end{tikzpicture}
&
\hspace{1.7cm}
\begin{tikzpicture}[thick,scale=0.5]
	\node[red, circle, fill, inner sep=0pt, outer sep=0pt] (1) at (1,0){};
	\node[red, circle, fill, inner sep=0pt, outer sep=0pt] (2) at (1,-1){};
	\node[red, circle, fill, inner sep=0pt, outer sep=0pt] (3) at (0,-2){};
	\node[red, circle, fill, inner sep=0pt, outer sep=0pt] (4) at (2,-2){};
	\node[red, circle, fill, inner sep=0pt, outer sep=0pt] (5) at (0,-3){};
	\node[red, circle, fill, inner sep=0pt, outer sep=0pt] (6) at (2,-3){};
	\node[inner sep=0pt, outer sep=0pt] (7) at (6,-1.3){$\Delta: F \rightarrow F \otimes F$};
	
	\draw [ultra thick, red, rounded corners] (1) -- (2);
	\draw [ultra thick, red, rounded corners] (2) -- (3);
	\draw [ultra thick, red, rounded corners] (2) -- (4);
	\draw [ultra thick, red, rounded corners] (3) -- (5);
	\draw [ultra thick, red, rounded corners] (4) -- (6);
\end{tikzpicture}
\\
\\
\begin{tikzpicture}[thick,scale=0.5]
	\node[red, circle, fill, inner sep=0pt, outer sep=0pt] (1) at (0,0){};
	\node[circle, fill, inner sep=0pt, outer sep=0pt] (2) at (2,0){};
	\node[red, circle, fill, inner sep=0pt, outer sep=0pt] (3) at (0,-1){};
	\node[circle, fill, inner sep=0pt, outer sep=0pt] (4) at (2,-1){};
	\node[circle, fill, inner sep=0pt, outer sep=0pt] (5) at (1,-2){};
	\node[circle, fill, inner sep=0pt, outer sep=0pt] (6) at (1,-3){};
	\node[inner sep=0pt, outer sep=0pt] (7) at (6,-1.3){$m: F \otimes M \rightarrow M$};
	
	\draw [ultra thick, red, rounded corners] (1) -- (3);
	\draw [ultra thin, rounded corners] (2) -- (4);
	\draw [ultra thick, red, rounded corners] (3) -- (5);
	\draw [ultra thin, rounded corners] (4) -- (5);
	\draw [ultra thin, rounded corners] (5) -- (6);
\end{tikzpicture}
&
\hspace{1.7cm}
\begin{tikzpicture}[thick,scale=0.5]
	\node[circle, fill, inner sep=0pt, outer sep=0pt] (1) at (1,0){};
	\node[circle, fill, inner sep=0pt, outer sep=0pt] (2) at (1,-1){};
	\node[red, circle, fill, inner sep=0pt, outer sep=0pt] (3) at (0,-2){};
	\node[circle, fill, inner sep=0pt, outer sep=0pt] (4) at (2,-2){};
	\node[red, circle, fill, inner sep=0pt, outer sep=0pt] (5) at (0,-3){};
	\node[circle, fill, inner sep=0pt, outer sep=0pt] (6) at (2,-3){};
	\node[inner sep=0pt, outer sep=0pt] (7) at (6,-1.3){$\delta: M \rightarrow F \otimes M$};
	
	\draw [ultra thin, rounded corners] (1) -- (2);
	\draw [ultra thick, red, rounded corners] (2) -- (3);
	\draw [ultra thin, rounded corners] (2) -- (4);
	\draw [ultra thick, red, rounded corners] (3) -- (5);
	\draw [ultra thin, rounded corners] (4) -- (6);
\end{tikzpicture}
\\
\\
\begin{tikzpicture}[thick,scale=0.5]
	\node[circle, fill, inner sep=0pt, outer sep=0pt] (1) at (0,0){};
	\node[red, circle, fill, inner sep=0pt, outer sep=0pt] (2) at (2,0){};
	\node[circle, fill, inner sep=0pt, outer sep=0pt] (3) at (0,-1){};
	\node[red, circle, fill, inner sep=0pt, outer sep=0pt] (4) at (2,-1){};
	\node[circle, fill, inner sep=0pt, outer sep=0pt] (5) at (1,-2){};
	\node[circle, fill, inner sep=0pt, outer sep=0pt] (6) at (1,-3){};
	\node[inner sep=0pt, outer sep=0pt] (7) at (6,-1.3){\hspace{7mm} $\delta^* = m: M^* \otimes F \rightarrow M^*$};
	
	\draw [ultra thin, rounded corners] (1) -- (3);
	\draw [ultra thick, red, rounded corners] (2) -- (4);
	\draw [ultra thin, rounded corners] (3) -- (5);
	\draw [ultra thick, red, rounded corners] (4) -- (5);
	\draw [ultra thin, rounded corners] (5) -- (6);
\end{tikzpicture}
&
\hspace{1.7cm}
\begin{tikzpicture}[thick,scale=0.5]
	\node[circle, fill, inner sep=0pt, outer sep=0pt] (1) at (1,0){};
	\node[circle, fill, inner sep=0pt, outer sep=0pt] (2) at (1,-1){};
	\node[circle, fill, inner sep=0pt, outer sep=0pt] (3) at (0,-2){};
	\node[red, circle, fill, inner sep=0pt, outer sep=0pt] (4) at (2,-2){};
	\node[circle, fill, inner sep=0pt, outer sep=0pt] (5) at (0,-3){};
	\node[red, circle, fill, inner sep=0pt, outer sep=0pt] (6) at (2,-3){};
	\node[inner sep=0pt, outer sep=0pt] (7) at (6,-1.3){\hspace{7mm} $m^* = \delta: M^* \rightarrow M^* \otimes F$};
	
	\draw [ultra thin, rounded corners] (1) -- (2);
	\draw [ultra thin, rounded corners] (2) -- (3);
	\draw [ultra thick, red, rounded corners] (2) -- (4);
	\draw [ultra thin, rounded corners] (3) -- (5);
	\draw [ultra thick, red, rounded corners] (4) -- (6);
\end{tikzpicture}.
\end{tabular}
\end{center}

Notice that we have drawn the red arcs representing the identity maps on the Frobenius algebra object with greater thickness than the black arcs representing the module, in keeping with our argument that formally the bands of Hughes et al. \cite{HKM19, HKM21} (when running vertically) should be (identity arrows on) a Frobenius algebra object in a ribbon category.

The foregoing discussion, together with the arguments from Petit \cite{P2008} has almost established our main theorem.  It remains only to discuss the categorical conditions which give invariance under the cup, cap, surface slide and swimming moves, which we explicitly state as part of the theorem:

\begin{theorem}
If ${\mathcal C} \supset {\mathcal B}$ is nested pair of ribbon fusion categories, such that $\Delta_{\mathcal C}$, $\Delta_{\mathcal B}$ and $\Delta_{{\mathcal C},{\mathcal B}}^{\prime \prime}$ are non-zero, and $(F, \mu, e, \Delta, \epsilon)$ is a symmetric Frobenius algebra object in $\mathcal B$, $(M, m)$ a left $F$-module in $\mathcal C$, with induced coaction $\delta := 1_F\otimes m(\Delta(e))\otimes 1_M$, moreover satisfying

\begin{itemize}
    \item[cap] there exists a scalar $k$ such that 
    \[ m(\delta) = k\cdot 1_M \]
    
    \item[cup] there exits a scalar $\kappa$ such that 
    \[ 1_M\otimes ev_M(m\otimes m \otimes 1_{M^*}(1_M\otimes coev_M)) = \kappa\cdot 1_M \]
    
    \item[swim] the image of 
    \[(m \otimes m)(1 \otimes \Delta(1_F) \otimes 1):M^* \otimes  M \rightarrow M^*\otimes M \] 
    
    is ${\mathcal B}$-transparent. 
    
    \end{itemize}
    
    Then the evaluation resulting from coloring a based Kirby diagram with banded unlink by
    
    \begin{itemize}
        \item coloring every two-handle attaching curve with $B_{\mathcal B}$, and inserting the map ${\mathsf d}_{B_{\mathcal B}}$ at the base point ,
        \item coloring every one-handle attaching curve with $B_{\mathcal C}$, and inserting the map ${\mathsf d}_{B_{\mathcal C}}$ at the base point ,
        \item coloring each band with $F$
        \item coloring each unlink component with $M$ on the downward sections and $M^*$ on the upward sections
        
    \end{itemize}
    
    \noindent and interpreting the meetings of bands and unlink components as the action, when normalized by dividing by $\Delta_{\mathcal B}^{b_0} (\Delta_{\mathcal{C}} \, \Delta_{\mathcal{C}, \mathcal{B}}'')^{b_+} k^s \kappa^\omega$, where $b_0$ is the multiplicity of 0 as an eigenvalue in the linking matrix of the framed link describing the 1- and 2-handles, $b_+$ is the number of positive eigenvalues in the linking matrix of the framed link describing the 1- and 2-handles, $s$ is the number of self-bands (bands with both feet on the same components of the unknot) and $\omega$ is the number of other-bands (bands with feet on two different components of the unknot),
    is an invariant of the surface-(0,1,2)-handlebody pair presented by the diagram.
\end{theorem}

\begin{proof}
Invariance under the band slide move was discussed in our explanation of why a Frobenius algebra object and a module over it are the appropriate coloring data for the bands and unlink components.  

Invariance under handle slide moves follows from the lemma of B\"{a}renz and Barrett \cite{BB18} discussed previous section, which extends easily to give the surface slide moves when the band is an object in $\mathcal B$.

The conditions cap and cup, together with the specified normalizations plainly render the normalized evaluation invariant under the cap and cup moves.

Finally, invariance under the band and 2-handle swimming moves follows from the $\mathcal B$-transparency of the image of the ``action in the middle of the copairing'', as the map given by left-hand side of the band swim and band/2-handle swim moves, regarded as string diagrams are unchanged by inserting a projection from $M^*\otimes M$ onto the image, followed by the inclusion of the image just below where the legs of the copairing act on the dual module and module.
\end{proof}

Notice that we have not claimed that the normalized evaluation is an invariant of the surface-closed-4-manifold pair presented by the diagram, though for consistency with Petit \cite{P2008} we have retained the normalization factor which makes the quantity invariant under the removal of a 2-handle cancellable by the addition of a 3-handle.  

This is because the evaluation is not invariant under the removal of 3-handle-cancellable 2-handles from which some part of the surface cannot be removed by ambient isotopy. We illustrate this lack of invariance in one of our examples in the next section.

\section{Examples from group theoretic fusion categories}

Having established our main theorem, it is incumbent upon us to supply at least some class of examples of the necessary algebraic data.  In this section will discuss one such class of examples and do some sample calculations of the resulting invariants. There are some unsatisfactory aspects to the class of initial data we explicitly describe, which we will discuss in our concluding section on directions for future research.

We begin with a standard result from the theory of fusion categories:

\begin{proposition}
For any finite abelian group $G$, together with a bicharacter $\beta:G\times G \rightarrow \mathbf{k}^\times$, and a 3-cocycle $a: G \times G \times G \rightarrow \mathbf{k}^\times$  category $\mathbf{Vec}_G$ of $G$-graded vector spaces over $\mathbf{k}$ is a ribbon category, with duals of homogeneous objects given by the dual vector space in the inverse degree, associator given by multiplication by $a$ on triples of homogeneous objects, braiding given on homogeneous objects of degrees $g$ and $\gamma$ by $\beta(g,\gamma)\cdot tw$, where $tw$ is the usual symmetry for tensor product of vector spaces, and the ribbon twist given on objects homogeneous of degree $g$ by multiplication by $\beta(g,g)$.  
\end{proposition}

We thus have a large class of categories from which  Petit's construction \cite{P2008} gives rise to invariants of smooth 4-manifolds:  $\mathbf{Vec}_H \subset \mathbf{Vec}_G$ for any group-subgroup pair $G \supset H$, with associator given by a 3-cocycle on $G$ and braiding and ribbon structure induced by a bicharacter on $G$.  Moreover, as all of the simple objects are invertible, their quantum dimensions are all 1, so that $\Delta_{\mathbf{Vec}_\Gamma} = |\Gamma|$ for $\Gamma = H, G$.  Also because of this, in this section we omit any mention of basepoints, as for these examples ${\mathsf d}$ is the identity natural transformation.

We wish to give examples of suitable Frobenius algebras in some of these and modules over them which satisfy the hypotheses of the main theorem, and thus give rise to invariants of 4-manifolds equipped with an embedded surface.


\begin{definition}
{\em A graded matrix algebra}  $M(n,\Gamma) $ in ${\mathbb Vec}_{A}$, for $A$ an abelian group is the object
\[ \bigoplus_{i,j =1}^n \mathbf{k}_{g_{ij}} \]

\noindent for $\Gamma = [g_{ij}]_{n\times n}$ a matrix of elements of $A$ satisfying
\[ \forall i,j,k \; g_{ij} + g_{jk} = g_{ij},  \]

\noindent equipped with matrix multiplication in the obvious sense as multiplication, unit the identity matrix $I$ (with the 1's on the diagonal all necessarily being of degree $0$).

It is a Frobenius algebra when equipped with the functional 
\[ tr = \sum_{i=1}^n p_{ii}: M(n,\Gamma) \rightarrow I=\mathbf{k}_{0} \]

\noindent which induces a pairing $(X,Y)\mapsto tr(XY)$, thereby inducing a copairing and the Frobenius comultiplication $\Delta$ (in terms of which the copairing is $\Delta(I)$), as happens when the classical Frobenius algebra structure on a matrix algebra is rewritten as an example of the categorical notion.
\end{definition}

To construct the desired example of a category-subcategory pair suitable for the Broda-Petit construction with a Frobenius algebra $F$ in the subcategory ${\mathcal B}$ and module over $M$ in the ambient category ${\mathcal C}$ for which none $F$, $M$ and $M^*\otimes M$ is either transparent or $\mathcal B$-transparent, but the image of
    \[(m \otimes m)1 \otimes \Delta(1_F) \otimes 1):M^* \otimes  M \rightarrow M^*\otimes M \]
    
\noindent is $\mathcal B$-transparent we consider an abelian group $G$ and a proper subgroup $H$ which contains an element $c \neq 0$

For any fixed $g\in G\setminus H$, the object $X_g = \mathbf{k}_{c+g}\oplus \mathbf{k}_g$, regarded as column vectors with lower entry of degree $g$, is then left module over $M(2,\Gamma_c)$, where 
\[ \Gamma_c := \left[ 
\begin{array}{cc}
1 & c \\ 
-c & 1 
\end{array} \right]. \]

Its dual object is $X_g^* = \mathbf{k}_{-c-g}\oplus \mathbf{k}_{-g}$, regarded as row vectors, and is a right module over $M(2,\Gamma_c)$.

By construction $F$ lies in $\mathbf{Vec}_H$ and $X_g$ does not. If $c$ and $g$ are chosen so that objects homogeneous of degrees $c$ and at least one of $g$ or $c+g$ are not transparent, our non-transparency desideratum is satisfied.  To see that we have invariance under the swimming moves we proceed as follows:

An easy calculation shows that

\[ \Delta(I) = \left[ 
\begin{array}{cc}
1 & 0 \\ 
0 & 0 
\end{array} \right] \otimes 
\left[ 
\begin{array}{cc}
1 & 0 \\ 
0 & 0 
\end{array} \right] +
\left[ 
\begin{array}{cc}
0 & 1 \\ 
0 & 0 
\end{array} \right] \otimes
\left[ 
\begin{array}{cc}
0 & 0 \\ 
1 & 0 
\end{array} \right]  \]

 \[ + \left[ 
\begin{array}{cc}
0 & 0 \\ 
1 & 0 
\end{array} \right] \otimes
\left[ 
\begin{array}{cc}
0 & 1 \\ 
0 & 0 
\end{array} \right] +
\left[ 
\begin{array}{cc}
0 & 0 \\ 
0 & 1 
\end{array} \right] \otimes
\left[ 
\begin{array}{cc}
0 & 0 \\ 
0 & 1 
\end{array} \right] \]

\noindent in which we have suppressed writing the degrees.  Recall the diagonal entries are of degree $0$, that in the upper right is of degree $c$ and that in the lower left is of degree $-c$.

We can thus calculate the image of this element acting on the left on $X_g^*$ and on the right on $X_g$, the object which must be transparent for the swimming moves to hold by calculating its action on generic elements:

{\footnotesize \[\left[ 
\begin{array}{cc}
x_{-c-g} & y_{-g}
\end{array} \right]\Bigg(\left[ 
\begin{array}{cc}
1 & 0 \\ 
0 & 0 
\end{array} \right] \otimes 
\left[ 
\begin{array}{cc}
1 & 0 \\ 
0 & 0 
\end{array} \right] +
\left[ 
\begin{array}{cc}
0 & 1 \\ 
0 & 0 
\end{array} \right] \otimes
\left[ 
\begin{array}{cc}
0 & 0 \\ 
1 & 0 
\end{array} \right]  \]

\[+ 
\left[ 
\begin{array}{cc}
0 & 0 \\ 
1 & 0 
\end{array} \right] \otimes
\left[ 
\begin{array}{cc}
0 & 1 \\ 
0 & 0 
\end{array} \right] 
+
\left[ 
\begin{array}{cc}
0 & 0 \\ 
0 & 1 
\end{array} \right] \otimes
\left[ 
\begin{array}{cc}
0 & 0 \\ 
0 & 1 
\end{array} \right]\Bigg) 
\left[ 
\begin{array}{c}
z_{c +g}  \\ 
w_{g}  
\end{array} \right]\]}

{\footnotesize \[ 
\left[ 
\begin{array}{cc}
x_{-c-g} & 0
\end{array} \right] \otimes
\left[ 
\begin{array}{c}
z_{c+g}  \\ 
0  
\end{array} \right] +
\left[ 
\begin{array}{cc}
0 & x_{-g}
\end{array} \right] \otimes
\left[ 
\begin{array}{c}
0  \\ 
z_{g}  
\end{array} \right] +
\left[ 
\begin{array}{cc}
y_{-c-g} & 0
\end{array} \right] \otimes
\left[ 
\begin{array}{c}
w_{c+g}  \\ 
0  
\end{array} \right] \]
\[
+ \left[ 
\begin{array}{cc}
0 & y_{-g}
\end{array} \right] \otimes
\left[ 
\begin{array}{c}
0  \\ 
w_{g}  
\end{array} \right]\]}
\smallskip

\noindent each summand of which is an element homogenous of degree $0$, thus showing the desired image is transparent, and thus {\em a fortiori} $\mathcal B$-transparent as required.

Similar calculations show that in this instance, $k$ and $\kappa$ are both 2.

A minimal example illustrating the desired features is obtained by letting $G = ({\mathbb Z}/\ell m, +)$, and $H$ be the copy of $({\mathbb Z}/\ell, +)$ generated by $m$.  This group has an obvious bicharacter $\beta(a,b) = \zeta^{ab}$ where $\zeta$ is a primitive $\ell m^{th}$-root of unity, thereby creating a braiding on ${\mathcal C} = {\bf Vec}_G$.  We use the trivial 3-cocycle to induce the (strict) associator.
Taking ${\mathcal B} = {\bf Vec}_H$,
choosing  $c \in \{m, 2m, 3m, \ldots (\ell-1)m\}$ and $g \in {\mathbb Z}/\ell m \setminus \{0, m, 2m, \ldots (\ell-1)m\}$ and using $F=M(2,\gamma_c)$ and $M = X_g$ give us our desiderata.

Calculating we find that $\Delta_{\mathcal B} = \ell$, $\Delta_{\mathcal C}\Delta_{{\mathcal B},{\mathcal C}}^{\prime \prime} = m^2\ell$

We now give some sample calculations of our invariants for this sort of initial data, concluding each with actual numerical values in the special case $\ell = 3, m = 2, c = 2,\zeta = \frac{1}{2} + \frac{\sqrt{3}}{2}i$ and $g =1$ 

The 4-sphere has invariant 1 (being given by an empty Kirby diagram), while the 4-sphere with 2-unlink of $n$ components (represented by the bandless unknot of $n$ components) has invariant $2^n$ (2 being the quantum dimension of the module $M$ used to color the unlink representing the surface).  A 4-sphere with a trivally embedded torus, whose banded unknot presentation $L_{\rm torus}$ is given below
\[\includegraphics[width=1in]{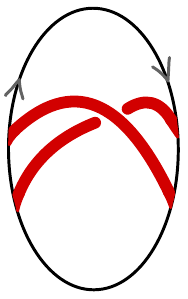}\]
\noindent has invariant $\frac{1}{2}$. A tedious calculation shows that $\big <L_{\rm torus} \big>$ is 2, so dividing by $k^2$ for the two self-bands gives $\frac{1}{2}$. 

Calculating the invariant for the simplest knotted 2-sphere in $S^4$, the spun trefoil, which has banded unknot presentation given below, 
\[\includegraphics[width=1.7in]{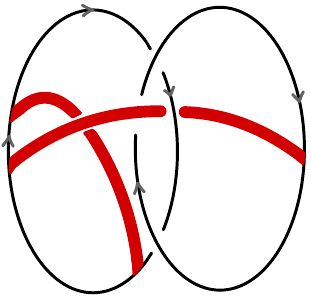}\]
\noindent gives 1, distinguishing it from the unknotted 2-sphere, whose invariant was 2.

We can calculate the invariant for ${\mathbb C}{\mathbb P}^2$, presented as an unknot with a single maximum, single minimum and a negative crossing as follows:

\[\includegraphics[width=1.7in]{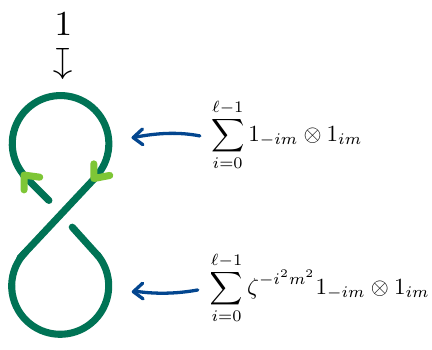}\]
\begin{eqnarray*}
1 & \mapsto & \sum_{i=0}^{\ell-1} 1_{-im} \otimes 1_{im} \\
 & \mapsto & \sum_{i=0}^{\ell-1} \zeta^{-i^2m^2} 1_{-im} \otimes 1_{im} \\
 & \mapsto & \sum_{i=0}^{\ell-1} \zeta^{-i^2m^2}
\end{eqnarray*}

Taking the particular case of $\ell = 3, m = 2, \zeta = \frac{1}{2} + \frac{\sqrt{3}}{2}i$ this gives $1 + \zeta^{-4} + \zeta^{-16} = 1 + 2\zeta^2 = \sqrt{3}i$.  Note, that as the only eigenvalue of the linking matrix in this case is $-1$ the normalization factor is 1.

It is easy to see that the value of $\big< - \big>$ on the diagram with the crossing reversed to present $\overline{{\mathbb C}{\mathbb P}}^2$ is the complex conjugate, and thus with Petit's normalization that $I_+(\overline{{\mathbb C}{\mathbb P}}^2) = \frac{1}{3}\sum_{i=0}^{\ell-1} \zeta^{-i^2m^2}$, which in our chosen particular case gives $-\frac{\sqrt{3}}{3}i$.  

Now, consider a sphere embedded in ${\mathbb C}{\mathbb P}^2$ (or rather, the (0,2)-handlebody completion of which by adding a 4-handle gives ${\mathbb C}{\mathbb P}^2$), which represents the homology class $nE$ where E is the exceptional divisor, that is the class of the sphere at infinity, and $n\in {\mathbb Z}$.  In the banded-unknot presentation of Hughes et. al \cite{HKM19} this is given by a single (bandless) unlink component linking the (-1)-framed 2-handle attaching curve of the simplest Kirby diagram for ${\mathbb C}{\mathbb P}^2$ with linking number $n$.

Calculating as above gives
\[\includegraphics[width=1.1in]{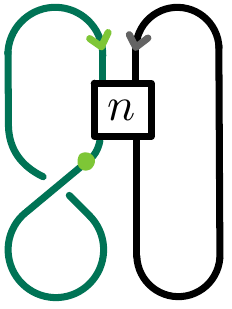}\]
\begin{eqnarray*}
1 & \mapsto & \left( \sum_{i=0}^{\ell-1} 1_{-im}\otimes 1_{im}\right)\otimes\left( \left[ \begin{array}{c} 1_{c+g} \\ 0 \end{array} \right]\otimes \left[ \begin{array}{cc} 1_{-c-g} & 0 \end{array} \right] + \left[ \begin{array}{c} 0 \\ 1_g \end{array} \right]\otimes \left[ \begin{array}{cc} 0 & 1_{-g} \end{array} \right]\right) \\
 & = & \left( \sum_{i=0}^{\ell-1} 1_{-im}\otimes 1_{im}\otimes \left[ \begin{array}{c} 1_{c+g} \\ 0 \end{array} \right]\otimes \left[ \begin{array}{cc} 1_{-c-g} & 0 \end{array} \right]\right) + \\
 & & \;\;\;\;\;\;\;\;\;\;\;\ \left( \sum_{i=0}^{\ell-1} 1_{-im}\otimes 1_{im}\otimes \left[ \begin{array}{c} 0 \\ 1_g \end{array} \right]\otimes \left[ \begin{array}{cc} 0 & 1_{-g} \end{array} \right]\right) \\
 & \mapsto & \left( \sum_{i=0}^{\ell-1} \zeta^{-i^2m^2 + 2ink(c+g)} 1_{-im}\otimes 1_{im}\otimes \left[ \begin{array}{c} 1_{c+g} \\ 0 \end{array} \right]\otimes \left[ \begin{array}{cc} 1_{-c-g} & 0 \end{array} \right]\right) + \\
 & & \;\;\;\;\;\;\;\;\;\;\;\ \left( \sum_{i=0}^{\ell-1} \zeta^{-i^2m^2 + 2inkg} 1_{-im}\otimes 1_{im}\otimes \left[ \begin{array}{c} 0 \\ 1_g \end{array} \right]\otimes \left[ \begin{array}{cc} 0 & 1_{-g} \end{array} \right]\right) \\
  & \mapsto &  \sum_{i=0}^{\ell-1} \zeta^{-i^2m^2 + 2ink(c+g)} + \zeta^{-i^2m^2 + 2inkg}
 \end{eqnarray*}
 
 Thus the value is a sum of two generalized quadratic Gauss sums. In general, if $c$ and $g$ are chosen correctly, the generalized Gauss sums will pick out the multiples of the exceptional divisor which are 0 modulo $\ell$.
 
 Taking the particular case of $\ell = 3, m = 2, c = 2, \zeta = \frac{1}{2} + \frac{\sqrt{3}}{2}i$ and $g =1$ gives $\frac{3}{2} + \frac{\sqrt{3}}{2}i$ if $n\neq 0$ mod 3, and $2\sqrt{3}i$ if $n=0$ mod 3.  Note in this case, since $b_0$, $b_+$, $s$ and $\omega$ are all zero the normalizing factors in the denominator are all 1.

 The example of the sphere $S^2\times \{0\} \subset S^2\times {\mathbb C}$ illustrates the lack of invariance under cancellation of a 2-handle which necessarily intersects all elements of the surface's ambient isotopy class in a (0,1,2)-handlebody.  The open 4-manifold is a handlebody with a single 0-handle and a single 2-handle attached on a 0-framed unknot.  The surface-4-manifold pair is presented by a Hopf link one component of which is the 2-handle attaching curve, and the other component of which is the unlink in a (band-free) banded unlink presentation of the sphere.

 Evaluating in our general example gives

 \[ \frac{1}{\ell}\sum_{j=0}^{\ell-1} \zeta^{-2(c+g)mj} + \zeta^{-2gmj}, \]

 \noindent which in our specific example with $\ell = 3, m = 2, c = 2, \zeta = \frac{1}{2} + \frac{\sqrt{3}}{2}i$ and $g =1$ gives 3 for the evaluation of the diagram and 1 for the invariant.  (The first summands in each term of the indexed summation are each 1, while the second summands add to 0).

As completing the handle-body to a closed 4-manifold by adjoining a 3-handle cancelling the 2-handle and a 4-handle to give the 4-sphere results in a trivially embedded 2-sphere, were our scalar also invariant under this move, the result should be the same as for an unlink consisting of a 0-framed unknotted 2-handle attaching curve and an unknot giving the(band-free) banded unlink presentation of the sphere, which as we noted before evaluates to 2 for any instance of our explicitly given initial data.

 \section{Directions for future research}
 
 Several obvious directions for future research arise from the present paper. 

 Most basic is the calculation of the invariants explicitly presented here for more substantial examples.

 Beyond the scope of this paper, but worth immediate investigation is whether there is a way of normalizing our invariants, using only the combinatorial data of a Kirby diagram with banded unlink, so that they become invariant under 2-3-handle cancellations in which the 2-handle intersects a part of the surface that cannot be isotoped out of the handle?  If this can be done, the result would be give an invariant of smooth surface-closed-4-manifold pairs.  If such a renormalization is available, it would then be worthwhile to find further renormalization that is insensitive to the addition of trivally embedded handles to the surface, thus giving rise to a functional on the second homology of the 4-manifold, which would be an invariant of the 4-manifold itself. Although we suspect Reutter's \cite{Reu20} ``no go'' theorem would extend to such an invariant, it is not immediately obvious that it does.

 Others involve seeking better initial data. 

 In particular, can our constructions be carried out in the non-semisimple setting provided by Costantino et al. \cite{CGHP23}?
 
But, even remaining in Petit's semisimple setting, are there examples in which the image of the copairing acting between the dual module and the module is not transparent, but only ${\mathcal B}$-transparent?  If the image of the ``copairing acting in the middle'' is transparent in the ambient category $\mathcal C$, rather than only $\mathcal B$-transparent, as in the examples we have explicitly constructed herein, the invariant will also be unchanged by the ungeometric ``band swim'' moves of 1-handle attaching curves and components of the surface-encoding unlink.  This is disappointing, as invariance under geometrically nonsensical moves generally weakens the topological significance of an invariant. Initial data in which the image is only ${\mathcal B}$-transparent, but not transparent, would not be invariant under these non-geometric moves and thus would be expected to yield more sensitive topological invariants.

Of course, as always for any new ``quantum'' topological invariant, there is the vexing question of giving a sensible geometric interpretation of the information it provides.

 Finally, our fully functorial approach to ``Kirby coloring'' using an endomorphism, rather than a formal linear combination of simple objects, leads to the question of which ribbon categories have objects admitting suitable endomorphisms and actions by other objects so that the analogue of Lemma 3.3 of B\"{a}renz and Barrett \cite{BB18} gives rise to handle-slide invariance.  Objects admitting chromatic maps in the sense of Costantino et al. \cite{CGHP23} certainly are among these, and could conceivably provide a complete solution if the projectivity conditions required for their definition cannot in some way be dispensed with to give a more general construction.

\appendix
\section*{Addendum and Errata}
\addcontentsline{toc}{section}{Addendum and Errata}
In our paper ``Embedded surface invariants via the Broda--Petit construction'' \cite{LY}, we constructed invariants of embedded surfaces in a four dimensional (0,1,2)-handlebody, and extended the construction to what we termed ``boundary links'' under the same relations encoding isotopies of surfaces in Hughes, Kim, and Miller\cite{HKM}.  The extension gave different results for two unknots in the 3-dimensional boundary of an Akbulut cork, one carried to the other by the diffeomorphism of the boundary which extended to a homeomorphism of the whole handlebody, but not to a diffeomorphism, one of which was slice and encoded an embedded sphere, the other of which was not, and thus did not encode a surface.  Discussions with other mathematicians about the construction revealed several errata in the paper, clarified some aspect of the construction and highlighted the importance of understanding the geometric content of the ``slice isotopy'' relation up to which the extension to boundary links is invariant.

First, the simple errata:  On page 2550015-23 in the last line of the computation at the top of the page, the variable $k$ should be removed from the second summand in the exponent of the second sum, and the two summands parenthesize so that the large summation applies to both.  And in the last line of page 2550015-26, the word ``not'' should be removed.


Turning to the properly mathematical content of this note, discussions with David Auckly \cite{A} revealed that as used by us in the context of $(0,1,2)$-handlebodies, the encoding of surfaces admits some ambiguity.  Consider the following example due to Auckly of a Kirby diagram with banded unlink:  

\[\includegraphics[width=2in]{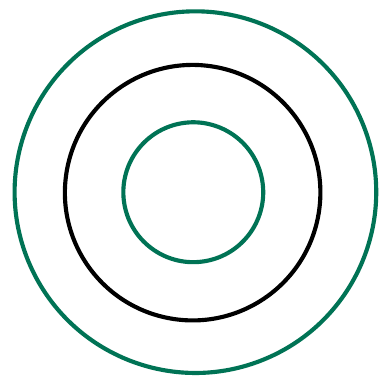}\]

There are two surfaces which can be obtained from the banded unlink, one of which goes over the 2-handle whose attaching curve is outermost in the diagram, the other of which goes over the other 2-handle.  These surfaces represent different homology classes in the $(0,1,2)$-handlebody, and thus are not isotopic.  In Hughes et al. \cite{HKM} the uniqueness of the surface encoded requires that the Kirby diagram be encoding a closed 4-manifold \cite{M}.  In this example adding 3-handles and a 4-handle gives for either of the surfaces in the $(0,1,2)$-handlebody an unknotted 2-sphere in the 4-sphere. 

In such circumstances, of course, our invariant will not be able to distinguish the two surfaces encodable by the same diagram.

The next pair of examples, suggested by Anna Beliakova \cite{B}, raises a number of questions about the extension to boundary links -- links lying in the intersection of the boundary of the handlebody and the boundary of the 0-handle.  Consider the Kirby diagram consisting of a Hopf link, one component of which is a 2-handle attaching curve, and the other component of which is dotted (indicating a location from which to bore out a 2-handle, thereby attaching a 1-handle).  The resulting handlebody is, of course, a 4-ball.  But if we consider a pair of boundary links, in one case linking the 2-handle attaching curve, in the other linking the dotted curve, as shown in the following

\[\includegraphics[width=4in]{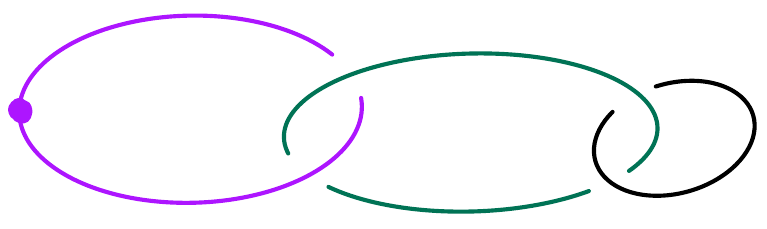}\]

\[\includegraphics[width=4in]{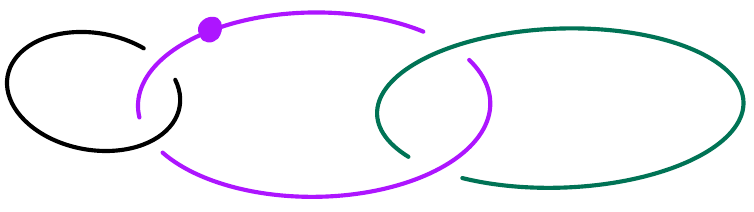}\]

\noindent The invariant of boundary links gives the quantum dimension of the color of the boundary link in the first case (since a slide over the dotted component unlinks the boundary link, and the normalization cancels the contribution of the Kirby diagram), while it gives 0 in the second case (since no summand of the tensor product of the coloring of the 2-handle curve with the coloring of the link lies in the M\"{u}ger center).  Obviously there is no distinction of smooth structure being made in this case, though our invariants behave as they did in our Akblut examples.

Discussions with Maggie Miller \cite{M} clarified for us the feature this is actually detecting:  one of the consequences of placing a surface in horizontal-vertical position (see \cite{HKM}) so that it can be represented by a banded unlink, is that the unlink of the banded unlink must be slice {\em in the $(0,1)$-handlebody}.  Of course, the boundary link in the second example above is not slice in the (0,1)-handlebody.

This reveals that, at least in the first instance, what is being detected in our examples from Akbulut is not that one boundary link is slice in the whole (0,1,2)-handlebody, while the other is not, but that one is slice in the (0,1)-handlebody, while the other is not.  This distinction alone seems insufficient to distinguish the smooth structures, leading to the following questions:

\begin{itemize}
    \item What is the geometric content of the equivalence relation on boundary links that we called slice isotopy -- equivalence under HKM moves \cite{HKM} applied to boundary links that do not necessarily encode a surface?
    \item Under what circumstances, if any, does non-sliceness in the (0,1)-handlebody imply non-sliceness in the (0,1,2)-handlebody?   
    \item Does there exist a pair of embedded surfaces in an Akbulut cork, related by a cork twist which are distinguished by any of our invariants?
\end{itemize}

An affirmative answer to the last question, or an answer to the second question which would apply to our examples from Akbulut would realize our hope that, despite being constructed from semisimple data, our invariants can catch a glimpse of exotic smooth structures.


\end{document}